\documentclass{amsart}

\usepackage{amsmath,amssymb}
\usepackage[latin1]{inputenc}
\usepackage[dvips]{graphics}
\input{xy}
\xyoption{poly}
\xyoption{2cell}
\xyoption{all}
\usepackage{epsfig}  
\usepackage{color}

\newcommand{\za}{\alpha}
\newcommand{\zb}{\beta}

\newcommand{\zg}{\gamma}

\newcommand{\zs}{\sigma}

\newcommand{\Tbar}{\overline{T}}
\newcommand{\Sbar}{\overline{S}}
\newcommand{\Mbar}{\overline{M}}
\newcommand{\pibar}{\overline{\pi}}
\newcommand{\zabar}{\overline{\alpha}}
\newcommand{\zbbar}{\overline{\beta}}
\newcommand{\zgbar}{\overline{\gamma}}

\newcommand{\taubar}{\overline{\tau}}

\newtheorem{thm}{Theorem}[section]
\newtheorem{prop}[thm]{Proposition}

\newtheorem{cor}[thm]{Corollary}
\newtheorem{lem}[thm]{Lemma}
\newtheorem{example}[thm]{Example}
\newtheorem{definition}{Definition} 
\newtheorem{rem}[thm]{Remark}

\newenvironment{pf}{{Proof}.}

\begin{document}
\title{On cluster algebras arising from unpunctured surfaces}
  \thanks{The first author is
  partly supported by the NSF grant DMS-0700358 and the University
  of Massachusetts at Amherst; the second author is partly supported
  by an NSERC  Discovery Grant} 
\author{Ralf Schiffler}
\author{Hugh Thomas}

\date{}


\begin{abstract} We study cluster algebras that are associated to unpunctured
surfaces, with coefficients arising from boundary arcs. We give a direct formula for the Laurent polynomial
expansion of cluster variables in these cluster algebras in terms of
certain paths on a triangulation of the surface.
 As an immediate consequence, we prove the positivity
conjecture of Fomin and Zelevinsky for these cluster algebras. In the
special case where the 
cluster algebra is acyclic, we also give a formula for the
expansion of cluster variables as a  polynomial whose indeterminates
are the cluster variables contained in the union of an arbitrary
acyclic cluster and all its neighbouring clusters in the mutation graph. 
 \end{abstract}
\maketitle



\begin{section}{Introduction}\label{section intro}
Cluster algebras, introduced in \cite{FZ1}, are commutative algebras
equipped with a distinguished set of generators, the \emph{cluster
  variables}. The cluster variables are grouped into sets of constant
cardinality $n$, the \emph{clusters}, and the integer $n$ is called
the \emph{rank} of the cluster algebra. Starting with an initial cluster
$\mathbf{x}$ (together with a skew symmetrizable integer $n\times n$ matrix $B=(b_{ij})$ and a coefficient vector
$\mathbf{p}=(p_i^\pm)$ whose entries are elements of a torsion-free
abelian group $\mathbb{P}$) the set of cluster variables is obtained
by repeated application of so called \emph{mutations}. 
To be more precise, let $\mathcal{F}$ be 
 the field of rational functions in the indeterminates $x_1,x_2,\ldots,x_n$
over the quotient field of the integer group ring
$\mathbb{ZP}$. Thus $\mathbf{x}=\{x_1,x_2,\ldots,x_n\}$ is a
transcendence basis for $\mathcal{F}$.
For every $k=1,2,\ldots,n$, the  mutation
$\mu_k(\mathbf{x})$ of the cluster
$\mathbf{x}=\{x_1,x_2,\ldots,x_n\}$ is a new cluster 
$\mu_k(\mathbf{x})=\mathbf{x}\setminus \{x_k\}\cup\{x_k'\}$ obtained
  from $\mathbf{x}$ by replacing the cluster variable $x_k$ by the new
  cluster variable 
\begin{equation}\label{intro 1}
x_k'= \frac{1}{x_k}\,\left(p_i^+\,\prod_{b_{ki}>0} x_i^{b_{ki}} +
p_i^-\,\prod_{b_{ki}<0} x_i^{-b_{ki}}\right)
\end{equation}  
in $\mathcal{F}$.
Mutations also change the attached matrix $B$ as well as the coefficient
vector $\mathbf{p}$, see \cite{FZ1}.

The set of all cluster variables is the union of all clusters obtained
from an initial cluster $\mathbf{x}$ by repeated mutations. Note that
this set may be infinite.

It is clear from the construction that every cluster variable is a
rational function in the initial cluster variables
$x_1,x_2,\ldots,x_n$. In \cite{FZ1} it is shown that every cluster
variable $u$ is actually a Laurent polynomial in the $x_i$, that is,
$u$ can be written as a reduced fraction 
\begin{equation}\label{intro 2}
u=\frac{f(x_1,x_2,\ldots,x_n)}{\prod_{i=1}^n x_i^{d_i}},
\end{equation}  
where $f\in\mathbb{ZP}[x_1,x_2,\ldots,x_n]$ and $d_i\ge 0$.
The right hand side of equation (\ref{intro 2}) is called the
\emph{cluster expansion} of $u$ in $\mathbf{x}$.

Inspired by the work of Fock and Goncharov \cite{FG1,FG2,FG3} and
Gekhtman, Shapiro and Vainshtein \cite{GSV1,GSV2} which discovered
cluster structures in the context of
Teichm\"uller theory, Fomin, Shapiro and Thurston  \cite{FST} initiated a
systematic study of the cluster algebras arising from 
triangulations of a surface with boundary and interior marked points.
 In this approach, clusters in the cluster algebra
correspond to triangulations of the surface. Our first main result is
a direct expansion formula  for cluster variables in cluster
 algebras associated to unpunctured
surfaces, with coefficients arising from boundary arcs, in terms of
certain paths on the triangulation, see Theorem \ref{thm 1}. 

As an immediate consequence, we
prove the positivity conjecture of Fomin and Zelevinsky \cite{FZ1} for
these cluster algebras, Corollary \ref{cor 1}.

For \emph{acyclic} cluster algebras, it has been shown in \cite{CA3} that if
the cluster $\mathbf{x}=\{x_1,x_2,\ldots,x_n\}$ occurs in an acyclic
seed and if $x_1',x_2',\ldots,x_n'$ denote the $n$ cluster variables
obtained by mutating $\mathbf{x}$ in each direction,
 then any cluster variable $u$ can be written as a
polynomial in $x_1,x_2,\ldots,x_n,x'_1,x'_2,\ldots,x'_n$, that is,
\begin{equation}\label{intro 3}
u=f(x_1,\ldots,x_n,x'_1,\ldots,x'_n) \qquad \textup{with }
f\in\mathbb{ZP}[x_1,\ldots,x_n,x'_1,\ldots,x'_n]. 
\end{equation}   
Our second main result, Theorem \ref{thm 2}, is an explicit formula
for this polynomial in the case of acyclic cluster algebras associated
to unpunctured surfaces, with coefficients arising from boundary arcs.

Theorem \ref{thm 1} has interesting intersections with work of other
people. In \cite{CCS2}, the authors obtained a formula for the
denominators of the cluster expansion in types $A,D$ and $E$, see also
\cite{BMR}. In \cite{CC,CK,CK2} an expansion formula was given in the
case where the cluster algebra is acyclic and the cluster lies in an
acyclic seed. Palu recently generalized this formula to arbitrary
clusters in an acyclic cluster algebra \cite{Palu}. All these formulas
use the cluster category introduced in \cite{BMRRT}, and in \cite{CCS1} for
type $A$. The formula that we give in this paper  not only uses a
very different approach, it also covers a large variety of cluster
algebras (parametrized by the genus of the surface, the number of
boundary components and the number of marked points on the boundary) for which no
formula has been known so far.  The only surfaces that give rise to acyclic
cluster algebras are the polygon and the annulus  corresponding to the
types $A$ and $\tilde A$ respectively. The proof of the positivity
conjecture for arbitrary clusters is new even in the acyclic types; in
\cite{CK} and \cite{CR} the conjecture is shown only in the case where
the initial seed is acyclic. 

In \cite{SZ,CZ,Z,MP} cluster expansions for cluster algebras of
rank 2 are given, in \cite{Propp,CP,FZ3} the case $A$ is
considered and in \cite{M} a cluster expansion for cluster algebras of
finite type is given for clusters that lie in a bipartite seed.

The paper is organized as follows. In section \ref{sect FST}, we
recall the construction of \cite{FST}. We state our expansion formula and
give two examples in section \ref{sect 1}. The proof of the formula
given in section \ref{sect 2}. Section \ref{sect 3} is devoted to our
polynomial formula for acyclic clusters.

The work in this paper extends
 the work of the first author in \cite{Sch}, in which the expansion formula in
 the type A case was proved.

The authors would like to acknowledge helpful conversations with Gregg
Musiker, Robert Marsh, and Sergey Fomin, and useful comments from a
referee.
They would also like to acknowledge the hospitality of the Centre de
Recherches Math\'ematiques.

\end{section} 



\begin{section}{Cluster algebras from surfaces}\label{sect FST}
In this subsection, we recall the construction of \cite{FST} in the
case of surfaces without punctures.

Let $S$ be a connected oriented 2-dimensional Riemann surface with
boundary and $M$ a non-empty set of marked points in the closure of
$S$ with at least one marked point on each boundary component. The
pair $(S,M)$ is called \emph{bordered surface with marked points}. Marked
points in the interior of $S$ are called \emph{punctures}.  

In this paper we will only consider surfaces $(S,M)$ such that all
marked points lie on the boundary of $S$, and we will refer to $(S,M)$
simply by \emph{unpunctured surface}. 

We say that two curves in $S$ \emph{do not cross} if they do not intersect
each other except that endpoints may coincide.

\begin{definition}
An \emph{arc} $\zg$ in $(S,M)$ is a curve in $S$ such that 
\begin{itemize}
\item[(a)] the endpoints are in $M$,
\item[(b)] $\zg$ does not cross itself,
\item[(c)] the relative interior of $\zg$ is disjoint from $M$ and
  from the boundary of $S$,
\item[(d)] $\zg$ does not cut out a monogon or a digon. 
\end{itemize}   
\end{definition}     
 Curves that connect two
marked points and lie entirely on the boundary of $S$ without passing
through a third marked point are called \emph{boundary arcs}.
Hence an arc is a curve between two marked points, which does not
intersect itself nor the boundary except possibly at its endpoints and
which is not homotopic to a point or a boundary arc.

Each arc is considered up to isotopy inside the class of such curves.

For any two arcs $\zg,\zg'$ in $S$, let $e(\zg,\zg')$ be the minimal
number of crossings of $\zg$ and $\zg'$, that is, $e(\zg,\zg')$ is the
minimum of
the numbers of crossings of  arcs $\za$ and $\za'$, where $\za$ is
isotopic to $\zg$ and $\za'$ is isotopic to $\zg'$.
Two arcs $\zg,\zg'$ are called \emph{compatible} if $e(\zg,\zg')=0$. 
A \emph{triangulation} is a maximal collection of
compatible arcs together with all boundary arcs. 
The arcs of a 
triangulation cut the surface into \emph{triangles}.
Since $(S,M)$ is an unpunctured surface, the three sides of each
triangle are distinct (in contrast to the case of surfaces with
punctures).  Any triangulation  has
$n+m$ elements, $n$ of which  are arcs in $S$, and the remaining $m$
elements are boundary arcs. Note that the number of boundary arcs
is equal to the number of marked points.

\begin{prop}\label{prop rank}
The number $n$ of arcs in any triangulation is  given by the formula 
$n=6g+3b+m-6$,  where $g$ is the
genus of $S$, $b$ is the number of boundary components and $m=|M|$ is the
number of marked points. The number $n$ is called the \emph{rank} of $(S,M)$.
\end{prop}  
\begin{pf} \cite[2.10]{FST}
\qed
\end{pf}  

Note that $b> 0$ since the set $M$ is not empty.
 Table \ref{table 1} gives some examples of unpunctured surfaces.

\begin{table}
\begin{center}
  \begin{tabular}{ c | c | c || l  }
  \  b\ \  &\ \  g \ \   & \ \  m \ \  &\  surface \\ \hline
    1 & 0 & n+3 & \ polygon \\ 
    1 & 1 & n-3 & \ torus with disk removed \\
    1 & 2 & n-9 & \ genus 2 surface with disk removed \\\hline 
    2 & 0 & n & \ annulus\\
    2 & 1 & n-6 & \ torus with 2 disks removed \\ 
    2 & 2 & n-12 & \ genus 2 surface with 2 disks removed \\ \hline
    3 & 0 & n-3 & \ pair of pants \\ \\
  \end{tabular}
\end{center}
\caption{Examples of unpunctured surfaces}\label{table 1}
\end{table}

Following  \cite{FST}, we associate a cluster algebra
$\mathcal{A}(S,M)$ to the unpunctured surface $(S,M)$ as follows.
The coefficient semifield is taken to be the tropical semifield
$\textup{Trop}(x_{n+1},x_{n+2},\ldots,x_{n+m})$, which is a free
abelian group, written multiplicatively, with $m$ generators
$x_{n+1},x_{n+2},\ldots,x_{n+m}$, and with an auxiliary addition
which we do not need to refer to here.
 
 Choose any triangulation
$T$, let $\tau_1,\tau_2,\ldots,\tau_n$ be the $n$ arcs of
$T$ and  denote the  $m$ boundary
arcs of the surface by $\tau_{n+1},\tau_{n+2},\ldots,\tau_{n+m}$. 
Each of the boundary arcs is a side in
precisely one triangle of the triangulation $T$. 
For any triangle $\Delta$ in $T$ define a matrix 
$B^\Delta=(b^\Delta_{ij})_{1\le i\le m+n, 1\le j\le n}$  by
\[ b_{ij}^\Delta=\left\{
\begin{array}{ll}
1 & \textup{if $\tau_i$ and $\tau_j$ are sides of 
  $\Delta$  }\\ & \textup{ with  $\tau_j$ following $\tau_i$  in the clockwise
  order;}\\
-1 &  \textup{if $\tau_i$ and $\tau_j$ are sides of
  $\Delta$}, \\& \textup{ with  $\tau_j$ following $\tau_i$  in the
  counter-clockwise 
  order;}\\
0& \textup{otherwise.}
\end{array} \right. \]
Then define the matrix 
$\tilde B(T)=(b_{ij})_{1\le i\le n+m, 1\le j\le n}$  by
$b_{ij}=\sum_\Delta b_{ij}^\Delta$, where the sum is taken over all
triangles in $T$, and let $B(T)=(b_{ij})_{1\le i,j\le n}$ be the
principal part of $\tilde B(T)$.
The matrix $B(T)$ is skew-symmetric and each of its entries  $b_{ij}$ is either
$0,1,-1,2$, or $-2$. An example where $b_{ij}=2 $ is given in Figure
\ref{fig bij=2}.

\begin{figure}
\input{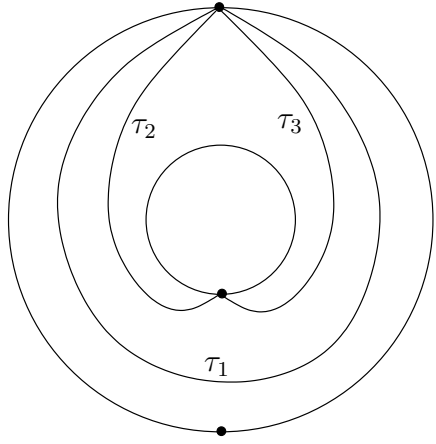}
\caption{A triangulation with $b_{23}=2$ \label{fig bij=2}}
\end{figure}   

Note that every arc $\tau$ can be in at most two triangles, since the
surface $S$ has no punctures. 
 
Let $\mathcal{A}(S,M)$ be the cluster algebra given by the seed
$(\mathbf{x},\mathbf{p},B(T))$ where
$\mathbf{x}=\{x_{\tau_1},x_{\tau_2},\ldots,x_{\tau_n}\}$ is the
cluster associated to the triangulation $T$, and the initial coefficient
vector $\mathbf{p}=(p_1^\pm,p_2^\pm,\ldots,p_n^\pm)\in (\textup{Trop}(x_{n+1},x_{n+2},\ldots,x_{n+m}))^{2n}$ is given by 
$$p_j^+=\prod_{i>n : b_{ij}=1} x_i \quad \textup{and}\quad
p_j^-=\prod_{i>n : b_{ij}=-1} x_i.$$ 

\begin{rem}
If one considers the cluster algebra $\mathcal{A}(S,M)$ with
 trivial coefficients then set $x_\tau=1$ for each boundary arc $\tau$.
\end{rem}


\end{section}

\begin{section}{Cluster expansions}\label{sect 1}
\begin{subsection}{$(T,\zg)$-paths}\label{sect 1.2}

A \emph{path} $\za$ in $S$ is a continuous function $\za:[0,1]\to S$.
Let $\za$ and $\zb$ be two paths in $S$, and let $x,y\in S$ be two points. Then we say that 
$\za$ and $\zb$ are \emph{homotopic between $x$ and $y$}, if 
there exist $s_1,s_2,t_1,t_2\in [0,1]$ such that $s_1<s_2,$ $t_1<t_2$,
$\za(s_1)=\zb(t_1)=x$, $\za(s_2)=\zb(t_2)=y$ and the restrictions
$\za\vert_{[s_1,s_2]}:[s_1,s_2]\to S$ and
$\zb\vert_{[t_1,t_2]}:[t_1,t_2]\to S$ are homotopic as paths from $x$
to $y$.

Let $T=\{\tau_1,\tau_2,\ldots,\tau_n,\tau_{n+1},\ldots,\tau_{N}\}$,
 with  $N=n+m$, be a
 triangulation of the unpunctured surface $(S,M)$,  where $\tau_1,\ldots,\tau_n$ are
arcs and $\tau_{n+1},\ldots,\tau_{N}$ are boundary arcs. 
Choose an orientation for each arc $\tau\in T$ and let $s(\tau)$ be
 the starting point of $\tau$ and $t(\tau)$ be its endpoint. Let
 $\tau^-$ be the arc $\tau$ with the opposite orientation. We will
 write $\tau^\pm$ if we want to consider both orientations at the same
 time. Let $\zg$ be an arc in $(S,M)$. Choose an orientation of $\zg$ and
 denote by $a$ its starting point and by $b$ its endpoint, thus $a,b\in M$.
Let $k=\sum_{\tau\in T}\,e(\zg,\tau)$ be the number of crossings
 between $\zg$ and $T$, and label the $k$ crossing points of $\zg$ and
 $T$ by $1,2,\ldots,k$ according to their order on $\zg$ such that $1$
 is the closest to $a$.

We will consider paths $\za$ in $S$ that are concatenations of arcs in the
triangulation $T$, more precisely, $\za=
(\za_1,\za_2,\ldots,\za_{\ell(\za)})$ with 
 $\za_i$ or $\za_i^- \in T$, for $i=1,2,\ldots, \ell(\za)$ and
$s(\za_1)=s(\zg)$, $t(\za_{\ell(\za)})=t(\zg)$, and  
  $s(\za_i)=t(\za_{i-1})$,   for  $i=2,\ldots,\ell(\za)$. We call such
a path a $T$-path. A $T$-path $\za$ is called \emph{reduced} if
$\za_i\ne \za_{i-1}^-$,   for  $i=2,\ldots,\ell(\za)$. 

\begin{definition}\label{Tpath}
A \emph{$(T,\zg)$-path} $\za$  is a reduced $T$-path
\[ \za = (\za_1,\za_2,\ldots,\za_{\ell(\za)})\]
such that 
\begin{itemize}
\item[\textup{(T1)}] $\ell(\za)$ is odd,
\item[\textup{(T2)}] if  $i$ is even, then $\za_i$ crosses $\zg$,
\item[\textup{(T3)}] for any $\tau\in T$, the number of even integers $i$ such that
  $\za_i=\tau^\pm$ is at most $e(\tau,\zg)$,
\item[\textup{(T4)}]  there exists a sequence
${\bf i}_\za=(i_0,i_2,i_4,\ldots,i_{\ell(\za)-1},i_{\ell(\za)+1})$,
  $i_j \in \{1,2,\ldots,k\}$, 
  such that $i_0=s(\zg), i_{\ell(\za)+1}=t(\zg)$ and
  $(i_2,i_4,\ldots,i_{\ell(\za)-1})$ is a sequence  of  labeled crossing points
   such that $i_j<i_\ell$ if $j<\ell$ and the 
  crossing point $i_j$ lies on $\za_j$, for
  $j=2,4,6,\ldots,\ell(\za)-1$, 
\item[\textup{(T5)}] for any two points $i_j,i_\ell$ in the
  sequence ${\bf i}_\za$ in \textup{(T4)}, with $j<\ell$, the paths $\za $
  and $\zg$ are homotopic between the points $i_j$ and $i_\ell$.
\end{itemize}    

\end{definition}

To any $(T,\zg)$-path $\za = (\za_1,\za_2,\ldots,\za_{\ell(\za)})$, 
we associate an element $x(\za)$ in the cluster algebra $\mathcal{A}(S,M)$ by

\begin{equation}\label{**}
x(\za)=
\prod_{i \textup{ odd}} x_{\za_i}
\ \prod_{i \textup{ even}}x_{\za_i}^{-1} . 
\end{equation}
Note that $x_{\tau}=x_{\tau^-}$.

\begin{definition}
Let $\mathcal{P}_T(\zg)$ denote the
set of $(T,\zg)$-paths.
\end{definition}     

\begin{prop}\label{prop 1} If $S $ is simply connected, then 
\begin{itemize}
\item[(a)] for any two arcs $\za$ and $\zb$ in $S$, we have
  $e(\za,\zb)\le 1$,
\item[(b)] any arc in $S$ can occur at most
once in any $(T,\zg)$-path, 
\item[(c)]  the map $\mathcal{P}_T(\zg)\to \mathcal{A}(S,M),\ \za\mapsto x(\za)$
is  injective.
\end{itemize}   
\end{prop}
\begin{pf}
Let $T$ be a triangulation, $\zg$ an arc and $\za \in
\mathcal{P}_T(\zg)$. 

(a) This follows directly from the definition of $e(\za,\zb)$ and the fact
that $S$ is simply connected.

(b) Suppose that $\za_j=\za_\ell^\pm$ with $j\ne \ell$.
Then $\za$ contains a loop $\za^\circ$ which runs through at least one crossing
point $i_t$ of $\zg$ and $T$, and $i_t$ lies  on an arc $\tau\in
T$. Statement (a) implies that $i_t$
is the only crossing point on $\tau$.
 The loop $\za^\circ$ is homotopically trivial, since $S$ is simply
connected. By condition (T5),  $\za$ and $\zg$ are homotopic between
$a=s(\zg)$ and $b=t(\zg)$, and then the isotopy class of  $\zg$ contains an arc which
does not cross $\tau$, hence $i_t$ is not a crossing point, a contradiction.
This shows (b).

(c) Suppose that $S$ is simply connected and $x(\alpha)= x(\alpha')$. Then
(b) implies that the set of even arcs and the set of odd arcs are
the same up to orientation in $\za$ and $\za'$.  
From condition (T4) it follows that the order of the even
arcs is the same.
  The even arcs divide $S$ into regions.  There is a unique odd arc
in the region between successive even arcs.  Therefore the order of
the odd arcs must be the same, and thus the order of the marked points
along the paths are the same.  By simply-connectedness, knowing the
sequence of
vertices determines the paths, and thus the two paths are the same.
\qed
\end{pf}

\end{subsection}

\begin{subsection}{Expansion formula}\label{sect main}
  The following theorem is the main result of this section. 

\begin{thm}\label{thm 1} Let $T$ be any
  triangulation of  an unpunctured surface $(S,M)$. Let $\zg$ be any arc
  in $(S,M)$ and let $x_\zg$ denote the corresponding cluster
  variable in $\mathcal{A}(S,M)$. Then

\begin{equation}\label{eqthm1}
  x_{\zg}=\sum_{\za\in\mathcal{P}_T(\zg)} x(\za).
\end{equation}
\end{thm}  

\begin{rem}\label{rem 2} \textup{If $S$ is simply connected, then
  \textup{Proposition \ref{prop 1}} implies that  
each $x(\za)$ is a reduced fraction whose denominator is a product of
cluster variables,
and that   each term   in the sum of equation 
  \textup{(\ref{eqthm1})}  appears with multiplicity one.}
\end{rem} 

The proof of Theorem \ref{thm 1} will be given in section \ref{sect 2}. To
illustrate the statement, we give two examples here.

\begin{example}{\bf The case $A_n$:} The cluster algebra
  $\mathcal{A}(S,M)$ is of type $A_n$ if $(S,M)$ is an
  $(n+3)$-gon. Our example illustrates the case $n=5$.
The following figure shows a triangulation $T=\{\tau_1,\ldots,\tau_{13}\}$ and
a (dotted) arc $\zg$. Next to it is a complete list
 of elements of $\mathcal{P}_T(\zg)$. 
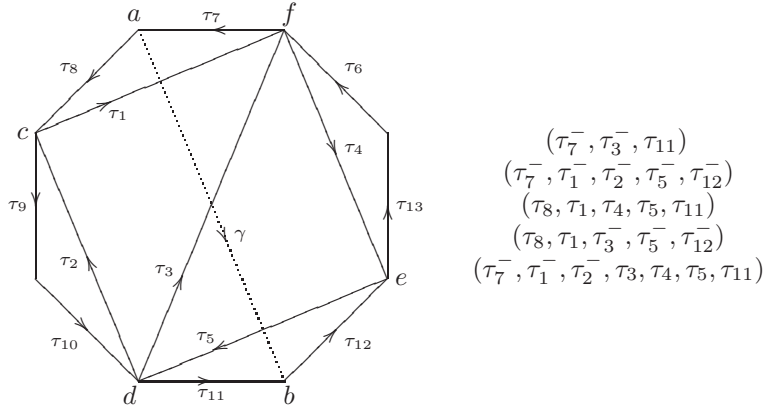
\[\begin{array}{cc}
\def\alphanum{\ifcase \xypolynode \or 6 \or 7\or 8 \or 9\or 10\or
  11\or 12\or 13\fi}
\xy/r6pc/: {\xypolygon8"A"{~<<{@{}}~><{@{-}|@{>}}
~>>{_{\tau_{\alphanum}}}}},
\POS"A2" \ar@{-}|(0.7)@{<}^(0.7){\tau_1} "A4",
\POS"A4" \ar@{-}|@{<}_(0.5){\tau_2} "A6",
\POS"A6" \ar@{-}|(0.3)@{>}^(0.3){\tau_3} "A2",
\POS"A2" \ar@{-}|@{>}^{\tau_4} "A8",
\POS"A6" \ar@{-}|(0.3)@{<}^(0.3){\tau_5} "A8",
\POS"A3" \ar@{.}|(0.6)@{>}^(0.6){\zg} "A7",
\POS"A3"\drop{\begin{array}{c}a\ \\ \\ \end{array}  }
\POS"A7"\drop{\begin{array}{c} \\ \ b \end{array}  } 
\POS"A2"\drop{\begin{array}{c}\  f\\ \\ \end{array}  } 
\POS"A4"\drop{\begin{array}{c} c\quad \\ \end{array}  }
  \POS"A6"\drop{\begin{array}{c} \\d \ \ \end{array}  }
  \POS"A8"\drop{\begin{array}{c} \quad e \end{array}  }  
\endxy
&

\begin{array}{c}
 (\tau^-_7,\tau^-_3,\tau_{11})\, 
\\( \tau^-_7,\tau^-_1,\tau^-_2,\tau^-_5,\tau^-_{12})\\
( \tau_8,\tau_1,\tau_4,\tau_5,\tau_{11})
\\( \tau_8,\tau_1,\tau^-_3,\tau^-_5,\tau^-_{12})\\
( \tau^-_7,\tau^-_1,\tau^-_2,\tau_3,\tau_4,\tau_5,\tau_{11}) 
\end{array} 
\end{array}  \]
Theorem \ref{thm
  1} thus implies that 
\[ x_\zg=
\frac{x_7x_{11}}{x_3} +
\frac{x_7x_2 x_{12}}{x_1x_5} +
\frac{x_8x_4 x_{11}}{x_1x_5} +
\frac{x_8x_3 x_{12}}{x_1x_5} +
\frac{x_7x_2x_4 x_{11}}{x_1x_3x_5} .
\]
\end{example}

\begin{example}{\bf The case $\tilde A_{n-1}$:}
 The cluster algebra
  $\mathcal{A}(S,M)$ is of type $\tilde A_{n-1}$ if $(S,M)$ is an
  annulus. Our example illustrates the case $n=4$. Figure
  \ref{figaffine} shows a triangulation
  $T=\{\tau_1,\tau_2,\ldots,\tau_8\}$ and a (dotted) arc $\zg$.
The complete list of elements of $\mathcal{P}_T(\zg) $ is as follows:
\[
\begin{array}{ll}
(\tau_5,\tau_1^-,\tau_8,\tau_3,\tau_4,\tau_1,\tau_2)
&(\tau_5,\tau_1^-,\tau_8,\tau_3,\tau_5,\tau_1^-,\tau_8)\\
(\tau_5,\tau_1^-,\tau_8,\tau_2^-,\tau_6,\tau_4,\tau_8)
&(\tau_5,\tau_1^-,\tau_8,\tau_2^-,\tau_6,\tau_3^-,\tau_7,\tau_1,\tau_2)\\
(\tau_5,\tau_1^-,\tau_8,\tau_2^-,\tau_6,\tau_3^-,\tau_7,\tau_4^-,\tau_5,\tau_1^-,\tau_8)
&(\tau_5,\tau_2,\tau_3,\tau_4,\tau_8)\\
(\tau_5,\tau_2,\tau_7,\tau_1,\tau_2)
&(\tau_5,\tau_2,\tau_7,\tau_4^-,\tau_5,\tau_1^-,\tau_8)\\
(\tau_4,\tau_1,\tau_2,\tau_3,\tau_4,\tau_1,\tau_2)
&(\tau_4,\tau_1,\tau_2,\tau_3,\tau_5,\tau_1^-,\tau_8)\\
(\tau_4,\tau_1,\tau_6,\tau_3^-,\tau_7,\tau_1,\tau_2)
&(\tau_4,\tau_1,\tau_6,\tau_3^-,\tau_7,\tau_4^-,\tau_5,\tau_1^-,\tau_8)\\
(\tau_4,\tau_1,\tau_6,\tau_4,\tau_8)\\
\end{array}  
\]
Hence Theorem \ref{thm 1} implies
\[
\begin{array}{rcl}
x_\zg &=& \displaystyle\frac{x_5x_8x_4x_2}{x_1x_3x_1}
+\frac{x_5x_8x_5x_8}{x_1x_3x_1}
+\frac{x_5x_8x_6x_8}{x_1x_2x_4}
+\frac{x_5x_8x_6x_7x_2}{x_1x_2x_3x_1}
+\frac{x_5x_8x_6x_7x_5x_8}{x_1x_2x_3x_4x_1}\\
\\
&& \displaystyle+\frac{x_5x_3x_8}{x_2x_4}
+\frac{x_5x_7x_2}{x_2x_1}
+\frac{x_5x_7x_5x_8}{x_2x_4x_1}
+\frac{x_4x_2x_4x_2}{x_1x_3x_1}
+\frac{x_4x_2x_5x_8}{x_1x_3x_1}
+\frac{x_4x_6x_7x_2}{x_1x_3x_1}\\
\\
&& \displaystyle
+\frac{x_4x_6x_7x_5x_8}{x_1x_3x_4x_1}
+\frac{x_4x_6x_8}{x_1x_4},
\end{array}  
\]
which can be simplified to
\[
\begin{array}{rcl}
x_\zg  &=& \displaystyle
2\frac{x_2x_4x_5x_8}{x_1^2x_3}
+\frac{x_5^2x_8^2}{x_1^2x_3}
+\frac{x_5x_6x_8^2}{x_1x_2x_4}
+2\frac{x_5x_6x_7x_8}{x_1^2x_3}
+\frac{x_5^2x_6x_7x_8^2}{x_1^2x_2x_3x_4}\\
\\
&& \displaystyle+\frac{x_3x_5x_8}{x_2x_4}
+\frac{x_5x_7}{x_1}
+\frac{x_5^2x_7x_8}{x_1x_2x_4}
+\frac{x_2^2x_4^2}{x_1^2x_3}
+\frac{x_2x_4x_6x_7}{x_1^2x_3}
+\frac{x_6x_8}{x_1}.
\end{array}  
\]

\begin{figure}
\begin{center}
\input{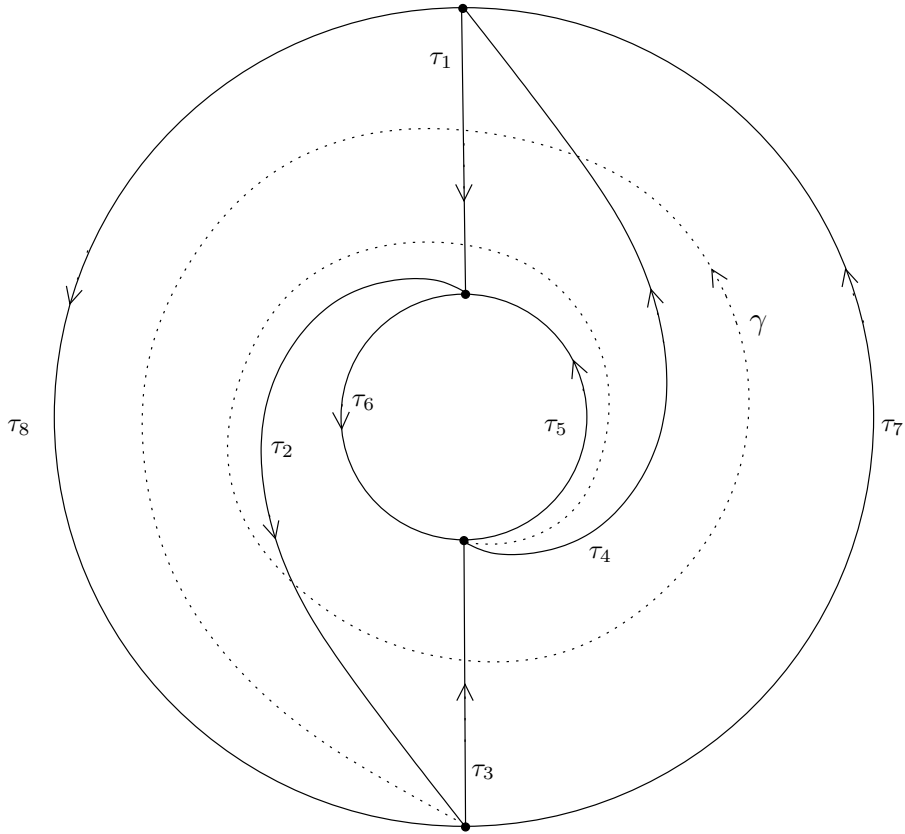}
\caption{The case $\tilde A_{n-1}$}\label{figaffine}
\end{center}
\end{figure}   
\end{example}

\end{subsection} 

\begin{subsection}{Positivity}
The following positivity
conjecture of \cite{FZ1} is a direct 
consequence of Theorem \ref{thm 1}.

\begin{cor}\label{cor 1} Let $(S,M) $ be an
  unpunctured surface. Let $x$ be any cluster variable in the cluster
 algebra $\mathcal{A}(S,M)$, and let 
 $\{x_1,\ldots,x_n\}$ be any cluster. Let
\[x=\frac{f(x_1,\ldots,x_n,x_{n+1},\ldots,x_{n+m})}{x_1^{d_1}\ldots x_n^{d_n}}\] be the
expansion of $x$ in the cluster $\{x_1,\ldots,x_n\}$, where $f$ is a
polynomial which is not divisible by any of the $x_1,\ldots,x_n$.
Then 
\begin{itemize}
\item[(a)] the coefficients of $f$ are non-negative integers,
\item[(b)] if $S$ is simply connected, the coefficients of $f$ are  either $0$ or $1$.
\end{itemize}   
\end{cor}  

\begin{pf} (a) is a direct consequence of Theorem \ref{thm 1}, and (b)
  follows from Remark \ref{rem 2}.
\qed
\end{pf}  

\end{subsection}
\end{section}

\begin{section}{Proof of Theorem \ref{thm 1}}\label{sect 2}

\begin{subsection}{The simply connected case}\label{sect simply connected}
Assume that $S$ is simply connected.
 Let $T=\{\tau_1,\ldots,\tau_N\}$, with  $N=n+m$, be a
triangulation of $S$  and let $\zg$ be an arc in $S$. Choose an
orientation of $\zg$ and let $a$ be its starting point and $b$ be its
endpoint. Suppose that $\zg\notin T$.
Proposition \ref{prop 1} implies that every arc in $T$ crosses $\zg$
 at most once.
 Among all arcs of $T$ that cross
$\zg$, there is a unique one, say $\tau_2$, such that its crossing point
with $\zg$ is the closest possible to the vertex $a$. 
Then there is a
unique triangle in $T$ having $\tau_2$ as one side and the vertex $a$
 as third point. Denote the other two sides of this triangle by
$\tau_1 $ and $\tau_3$ and let $c$ be the common endpoint of
$\tau_3$ and $\tau_2$, and  $d$  the common endpoint of
$\tau_1$ and $\tau_2$ (see Figure \ref{fig 3}). Note that $\tau_1,
  \tau_3$ may be boundary arcs.
\begin{figure}
\begin{center}
\input{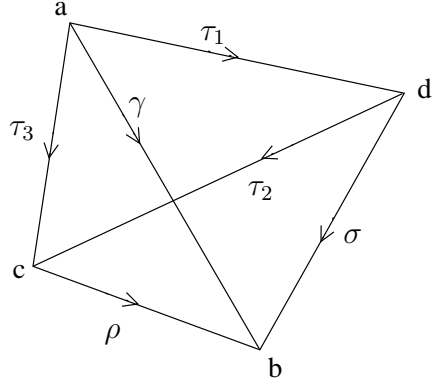}
\caption{Proof of Theorem \ref{thm 1}}\label{fig 3}
\end{center}
\end{figure}   
Now consider the unique quadrilateral  in which $\zg$ and
$\tau_2$ are the diagonals. Two of its sides are $\tau_1$ and
$\tau_3$. Denote the other two sides by $\rho$ and
$\zs$ in such a way that $\rho$ is the side opposite to $\tau_1$
(see Figure \ref{fig 3}). We assume without loss of generality that
the orientations of the arcs $\tau_1,\tau_2,\tau_3,\rho$ and $\zs$ are
as in Figure \ref{fig 3}. In particular
$s(\tau_1)=s(\tau_3)=a $ and  $s(\tau_2)=t(\tau_1)$. 
We will keep this setup for the rest of this
subsection.

\begin{lem}\label{lem 0}
\begin{itemize}
\item[(a)] If $\tau_i\in T$ crosses $\rho$ (respectively $\zs$), then $\tau_i$
  crosses $\zg$.
\item[(b)]  If $\tau_i\in T$ is incident to $a$, then $\tau_i$
  crosses neither $\rho$ nor $\zs$.
\item[(c)]  If $\tau_i\in T$ crosses $\zg$ and does not cross $\rho$
  (respectively $\zs$), then $\tau_i$ is incident to $c$ (respectively $d$).
\end{itemize}   
\end{lem}  

\begin{pf} Since $S$ is simply connected, this follows directly from
  the construction and the fact that $\tau_i$ does not cross $\tau_2$.
\qed
 \end{pf}
 
Let $\mathcal{P}_T(\zg)_{\tau_j}$ denote the subset of $\mathcal{P}_T(\zg)
$  of all 
$(T,\zg)$-paths $\za$ that start with the arc $\tau_j^\pm$ and let
$\mathcal{P}_T(\zg)_{-\tau_j}$ be the subset of $\mathcal{P}_T(\zg) $
 of all 
$(T,\zg)$-paths $\za$ that do not contain  the arc $\tau_j^\pm$.
Similarly, let $\mathcal{P}_T(\zg)_{\tau_j \tau_\ell}$ denote the subset of
$\mathcal{P}_T(\zg)_{\tau_j} $   of all
$(T,\zg)$-paths $\za$ that start with the arcs $\tau^\pm_j\tau^\pm_\ell$ and let
$\mathcal{P}_T(\zg)_{\tau_j,-\tau_\ell}$ be the subset of
$\mathcal{P}_T(\zg)_{\tau_j} $ of all   
$(T,\zg)$-paths $\za$ that start with the arc $\tau_j^\pm$ and do not contain
the arc $\tau_\ell^\pm$.

\begin{lem}\label{lem 5} We have
$\mathcal{P}_T(\zg)=\mathcal{P}_T(\zg)_{\tau_1}\sqcup
\mathcal{P}_T(\zg)_{\tau_3}$. 
\end{lem} 

\begin{pf}
Let $\za=({\za_1},{\za_2},\ldots,{\za_{\ell(\za)}})$ be an
arbitrary element of $\mathcal{P}_T(\zg)$. 
Then $s(\za_1)=a $, and $\za_2$ crosses $\zg$ at the
crossing point $i_2$  of ${\bf i}_\za$. Condition (T5) implies
that $\za$ and $\zg$ are homotopic between $a$ and $i_2$. On the other hand, $\zg$ crosses $\tau_{2}$ at
the crossing point $1\le i_2$, and thus,  $\za_1$ must intersect
$\tau_2$ either in the interior of $S$ or on the boundary. Therefore,
either $\za_2=\tau_2$ or $t(\za_1) \in\{c,d\}$, and in
both cases we must have  $\za_1\in\{\tau^\pm_1,\tau^\pm_3\}.$
\qed
\end{pf}

\begin{lem}\label{cor 2}
We have
\begin{itemize} 
\item[(a)] $\mathcal{P}_T(\zg)_{\tau_1}=\mathcal{P}_T(\zg)_{\tau_1\tau_2}
  \sqcup \mathcal{P}_T(\zg)_{\tau_1,-\tau_2}$,
\item[(b)] $\mathcal{P}_T(\zg)_{\tau_3}=\mathcal{P}_T(\zg)_{\tau_3\tau_2}
  \sqcup \mathcal{P}_T(\zg)_{\tau_3,-\tau_2}$.
\end{itemize}   
\end{lem}   %

\begin{pf} By construction, $\tau_2$ is the arc of the triangulation $T$ such
  that its crossing point with $\zg$ is the closest possible to the
  vertex $a$. Hence, the result follows from condition (T4).
\qed
\end{pf}  

Let $\rho$ be as above (see Figure \ref{fig 3}). We will construct two
maps $f:\mathcal{P}_T(\rho)_{\tau_2} \to \mathcal{P}_T(\zg) $ and
$g:\mathcal{P}_T(\rho)_{-\tau_2} \to \mathcal{P}_T(\zg).$

Let $\zb=(\zb_1,\ldots,\zb_{\ell(\zb)})$ be
any path in $\mathcal{P}_T(\rho)$.
Suppose first that $\zb_1=\tau^-_2$, thus $t(\zb_1)=d$. 
In this case, let $f(\zb)$ be the path in $\mathcal{P}_T(\zg)$
obtained from 
$\zb$ by replacing the first arc $\tau^-_2$ by $\tau_1$, that is 
\[f(\zb)=(\tau_1,\zb_2,\ldots,\zb_{\ell(\zb)}).\]
Suppose now that $\zb_j\ne \tau_2^\pm$ for all $j$. In this case, let $g(\zb)$ be the
composition of the paths $(\tau_1,\tau_2)$ and
$\zb$, that is 
\[g(\zb)=(\tau_1,\tau_2,\zb_1,\zb_2,\ldots,\zb_{\ell(\zb)}).\] 
Let us check that $f(\zb)$ and $g(\zb)$ are elements of
$\mathcal{P}_T(\zg)$.
 Indeed, the fact that $f(\zb)$ and $g(\zb)$ are $T$-paths is
 immediate; they are reduced by construction,
property (T1) is clear  and (T2)
follows from Lemma \ref{lem 0}(a). In order to show (T3), we need to
prove that $\zb_j\ne\tau_2^\pm$, for all even $j$; but this follows from
the fact that for even $j$, $\zb_j$ crosses $\rho$.
The condition (T4) holds  since the crossing point of  $\tau_2$ 
and $\zg$ is the closest possible to $a$, and (T5) holds, since
$S$ is simply connected.
We have the following lemma.
\begin{lem}\label{lem f}
The maps $f$ and $g$  induce  bijections 
\[
\begin{array}{ccccccccc}
 f:\mathcal{P}_T(\rho)_{\tau_2}\to
 \mathcal{P}_T(\zg)_{\tau_1,-\tau_2}&\quad and \quad&
 g:\mathcal{P}_T(\rho)_{-\tau_2}\to \mathcal{P}_T(\zg)_{\tau_1\tau_2},
\end{array}  \]
and 
\begin{equation}\label{eq 2}
  x(f(\zg))=\frac{x_{\tau_1}}{x_{\tau_2}} \, x(\zg) \quad \textup{\it and}\quad
 x(g(\zg))=\frac{x_{\tau_1}}{x_{\tau_2}} \, x(\zg) 
\end{equation}
\end{lem}  
\begin{pf} 
The formulas (\ref{eq 2}) follow directly from the definitions of $f$
and $g$. These formulas together with Proposition \ref{prop 1} imply the
injectivity of $f$ and $g$. 
To show the surjectivity of $f$, suppose that 
$\mathcal{P}_T(\zg)_{\tau_1,-\tau_2}$ is not empty and
let $\za\in \mathcal{P}_T(\zg)_{\tau_1,-\tau_2}$ be an arbitrary
element. Say 
\[\za=(\tau_1,\za_2,\za_3,\za_4,\ldots,\za_{\ell(\alpha)}).\]
We need to show that the path
\[\zb=(\tau^-_2,\za_2,\za_3,\za_4,\ldots,\za_{\ell(\alpha)})\]
is an element of $\mathcal{P}_T(\rho)_{\tau_2}$. It is a
reduced $T$-path  because the path $\za$
does not contain the arc $\tau_2^\pm$. Condition (T1)  holds since $\za\in
\mathcal{P}_T(\zg)$
 and condition (T5) holds since $S$ is simply connected. 
Moreover, since 
$\mathcal{P}_T(\zg)_{\tau_1,-\tau_2}$ is not empty,  there exists an arc
 in $T\setminus\{\tau_2\}$ which is incident to $d$ and
  crosses $\zg$. Since $T$ is a triangulation, it follows that any
  diagonal in  $T\setminus\{\tau_2\}$ that crosses $\zg$ also crosses
  $\rho$. Thus $\zb$ satisfies conditions (T2) and (T4), because  $\za\in
\mathcal{P}_T(\zg)$. Consequently, $\zb$ also satisfies condition
(T3), by Proposition \ref{prop 1} (b). 
This shows that $\zb\in \mathcal{P}_T(\rho)$, and since $\zb$ starts
with the arc $\tau_2^-$, we have  $\zb\in \mathcal{P}_T(\rho)_{\tau_2}$.
 Hence  $f$ is surjective.

It remains to show that $g$ is surjective. Let $\alpha \in
\mathcal{P}_T(\zg)_{\tau_1\tau_2}$ be arbitrary. Say 
\[\za=(\tau_1,\tau_2,\za_3,\za_4,\ldots,\za_{\ell(\alpha)}).\]
We have to show that
\[\zb=(\za_3,\za_4,\ldots,\za_{\ell(\alpha)})
\in \mathcal{P}_T(\rho).\]
It is a reduced $T$-path, since $\alpha \in
\mathcal{P}_T(\zg)$, condition (T1) follows since
$s(\za_3)=t(\tau_2)=c$,
 and condition (T5) holds because $S$ is simply
connected.  Let us show (T2).  
We need to show that any even arc of $\beta$ crosses $\rho$. Since
$\alpha\in\mathcal{P}_T(\zg)$, we know that every even arc of $\zb$
crosses $\zg$.  
Thus by Lemma \ref{lem 0}(c), if there is an even arc of $\zb$
that does not cross $\rho$, then  
this arc has to be incident to $c$.
Since $\zb$ starts at $c$, its first even arc $\za_{4}$ cannot  be 
incident to $c$, because $\za$ is reduced, and thus $\za_4$ crosses both $\zg$ and $\rho$. Then, since 
$\za $ satisfies (T4), every even arc of $\zb$ crosses $\zg$ and $\rho$. 
This shows (T2), and (T4) follows from the fact that $\za\in
\mathcal{P}_T(\zg)$. Condition (T3) holds  by Proposition \ref{prop 1}.
Hence $\zb\in \mathcal{P}_T(\rho)$ and $g$ is surjective.
\qed
\end{pf}

\begin{lem}\label{lem sum} We have
\[
\begin{array}{lcrclc}
(a)&\hspace*{45pt}& \displaystyle\sum_{\zb\in \mathcal{P}_T(\rho)  }
  x(\zg)\ \frac{x_{\tau_1}}{x_{\tau_2}} &=&   
\displaystyle \sum_{\za\in
 \mathcal{P}_T(\zg)_{\tau_1}  } x(\za)&\hspace*{50pt} \\ \\ 
(b)&& \displaystyle\sum_{\zb\in \mathcal{P}_T(\zs)  } x(\zg) \
  \frac{x_{\tau_3}}{x_{\tau_2}} &=& 
\displaystyle\sum_{\za\in 
 \mathcal{P}_T(\zg)_{\tau_3}  } x(\za).
\end{array}  
\]
\end{lem}  

\begin{pf}
The first statement follows from  Lemma \ref{cor 2}(a), Lemma
\ref{lem f} and the fact that 
 $\mathcal{P}_T(\rho) = \mathcal{P}_T(\rho)_{\tau_2}\sqcup
\mathcal{P}_T(\rho)_{-\tau_2}$. 
The second statement follows by symmetry.
\qed \\
\end{pf}

\emph{Proof of Theorem \ref{thm 1} for simply connected surfaces $S$.}
The total number
of crossings between $\zg$ and $T$ is $e(T,\zg)=\sum_{\tau_i\in
  T}\, e(\tau_i,\zg)$. 

We prove the theorem by induction on $e(T,\zg)$.
If $e(T,\zg) =0$, then $\zg\in T$. In this case, no element of $T$ crosses
$\zg$ and, by condition (T2), the set
$\mathcal{P}_T(\zg)$ contains exactly one element: $\mathcal{P}_T(\zg)=\{ \zg\}$. Thus 
\[\sum_{\za\in \mathcal{P}_T(\zg) }x(\za) = x(\zg) =
x_{\zg}.\]

Suppose now that $e(T,\zg)\ge 1$. 
As before, consider the unique quadrilateral in which $\zg$ and
$\tau_2$ are the diagonals (see Figure \ref{fig 3}).
Thus, in the cluster algebra $\mathcal{A}(S,M)$, we have the  exchange relation 
\begin{equation}\label{exch}
  x_\zg\,x_{\tau_2} = x_{\tau_1}\,x_\rho +x_{\tau_3}\,x_{\zs}. \end{equation} 
Moreover, any arc in $T$ that crosses $\rho$ (respectively $\zs$) also
crosses $\zg$, by Lemma \ref{lem 0}(a), and, moreover, $\tau_2$ crosses
$\zg$ but crosses neither $\rho$ nor $\zs$. Thus $e(T,\rho)<e(T,\zg)$ and $
e(T,\zs)<e(T,\zg)$, and by 
induction hypothesis
\[x_\rho=\sum_{\zb\in \mathcal{P}_T(\rho)} x(\zb)
\qquad \textup{and} \qquad
x_{\zs}=\sum_{\zb\in \mathcal{P}_T(\zs)} x(\zb).
\]
Therefore, we can write the exchange relation (\ref{exch}) as 
\begin{eqnarray}\nonumber
 x_\zg &=&\sum_{\zb\in  \mathcal{P}_T(\rho)} x(\zb) \ \frac{x_{\tau_1}}{x_{\tau_2}}
\quad+\sum_{\zb\in \mathcal{P}_T(\zs)} x(\zb)\ \frac{x_{\tau_3}}{x_{\tau_2}}.
\end{eqnarray}  
The theorem now follows from Lemma \ref{lem sum} and Lemma \ref{lem 5}.
\qed

\end{subsection}

\begin{subsection}{The non-simply connected case}\label{sect non
    simply connected}
In this subsection, we prove Theorem \ref{thm 1} for any unpunctured
surface $(S,M)$.
The key idea of the proof is to work in a universal cover of $S$, so that
we can use Theorem \ref{thm 1} for simply connected surfaces.
First, we prove a Lemma which will allow us to use induction later.

\begin{lem}\label{lem exchange}
Let $T=\{\tau_1,\ldots,\tau_N\}$  be a triangulation of $S$, and let
  ${\zb}$ be an arc in $S$ which is not in $T$.
 Let $e({\zb},T)$ be the number of crossings
  between ${\zb}$ and $T$. 
Then there exist five arcs
  $\rho_1,\,\rho_2,\,\zs_1,\,\zs_2$ and ${\zb'}$ in $S$ such that 
\begin{itemize}
\item[(a)] each of $\rho_1,\,\rho_2,\,\zs_1, \,\zs_2$  and ${\zb'}$ crosses $T$
  less than $e({\zb},T)$ times, 
\item[(b)] $\rho_1,\,\rho_2,\,\zs_1,\,\zs_2$ are the sides of a simply
  connected 
  quadrilateral in which ${\zb}$ and ${\zb'}$ are the diagonals,
\item[(c)] in the cluster algebra   $\mathcal{A}(S,M)$, we have the
  exchange relation 
\begin{equation}\label{eq lem exchange}
 x_{{\zb}}x_{{\zb'}} =  x_{\rho_1} x_{\rho_2}+ x_{\zs_1}x_{ \zs_2}.
\end{equation}  
\end{itemize}   
\end{lem}  

\begin{pf} We prove the Lemma by induction on $k=e({\zb},T)$. If
  $k=1$, then let ${\zb'}\in T$ be the unique arc that crosses
  ${\zb}$. Then there exists a unique quadrilateral in $T$ in which
  ${\zb'}$ and ${\zb}$ are the diagonals. Let $\rho_1$ and $ \rho_2$ be two
  opposite sides of this quadrilateral and  let $\zs_1$ and $ \zs_2$
  be the other two  opposite sides. These arcs satisfy (a),(b) and (c).

Suppose that $k\ge 2$. Choose an orientation of ${\zb}$ and denote its
starting point by $a$ and its endpoint by $b$ (note that $a$ and $b$
may be the same point). Label the $k$ crossing points of ${\zb}$ and $T$
by $1,2,\ldots,k$ according to their order on ${\zb}$, such that the
crossing point $1$ is the closest to $a$. Let $\tau$ be an arc of the
triangulation $T$ such that $e(\tau,{\zb})=r\ge 1.$ As before with
${\zb}$, label the $r$ crossing
points of $\tau$ and $\zb$ by $i_1,i_2,\ldots,i_r$ according to
their order on $\tau$ ! (see Figure \ref{figtau}). Thus $r\le k$,
$\{i_1,i_2,\ldots,i_r\}\subset\{1,2,\ldots,k\}$. Note that $j<\ell$
does \emph{not} imply $i_j<i_\ell$.

\begin{figure}
\input{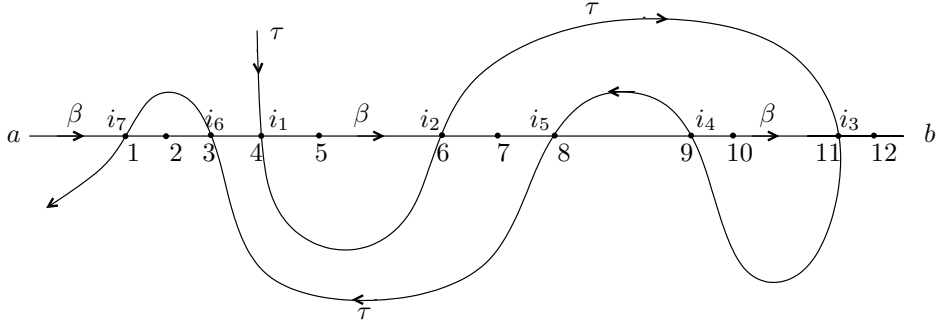}
\caption{Proof of Lemma \ref{lem exchange}, $k=12$ and $i_\ell=i_2$.}\label{figtau}
\end{figure}   

Using $\tau$ and ${\zb}$, we will now construct the five arcs of the
Lemma. Recall that ${\zb}^-$ (respectively  $\tau^-$) denotes the path ${\zb}$
(respectively $\tau$) with the opposite orientation.

Let $\ell$ be such that 

\begin{equation}\label{ell}
  \vert (k+1)/2-i_\ell\vert \ \le\ \vert (k+1)/2
-i_j\vert, \quad\textup{ for all }j=1,2,\ldots,r.
\end{equation}
In other words, among all $r$ crossing points of ${\zb}$ and $\tau$, the
point $i_\ell$ is midmost on ${\zb}$ with respect to the $k$ crossing
points of ${\zb}$ and $T$ (see Figure \ref{figtau} for an example).
We will distinguish four cases:
\begin{enumerate}
\item \emph{$i_{\ell-1}<i_\ell$ and $i_{\ell+1}<i_\ell$.} Consider the
  following   arcs (see Figure \ref{fig arcs}): 

Let ${\zb'}=(a,i_{\ell-1},i_{\ell+1},a\mid {\zb},\tau,{\zb}^-)$  be the arc
  that starts at at $a$ and is homotopic to ${\zb}$ 
  up to the crossing point $i_{\ell-1}$, then, from $i_{\ell-1}$ to
  $i_{\ell+1}$, ${\zb'} $ is homotopic to $\tau$, and from $i_{\ell+1}$ to
  $a$, ${\zb'}$ is homotopic to ${\zb^-}$. Note that ${\zb'}$ and ${\zb}$ cross
  exactly once, namely at the point $i_\ell$.

In a similar way, let
\[
\begin{array}{cc}
 \rho_1=(a,i_\ell,i_{\ell-1},a\mid {\zb},\tau^-,{\zb}^-) 
&\rho_2=(b,i_\ell,i_{\ell+1},a\mid {\zb}^-,\tau,{\zb}^-) \\
 \zs_1=(a,i_{\ell-1},i_{\ell},b\mid {\zb},\tau,{\zb})
&\zs_2=(a,i_{\ell+1},i_{\ell},a\mid {\zb},\tau^-,{\zb}^-)
\end{array}  
\]
In the special case where $\ell=1$, (respectively $l=r$), we define

\[
\begin{array}{cc}
{\zb'}=(c,i_{\ell+1},a\mid \tau, {\zb}^-) & (\textup{respectively } 
{\zb'}=(a,i_{\ell-1},d\mid {\zb}, \tau)\\
\rho_1=(a,i_{\ell},c\mid {\zb},\tau^-) & (\textup{respectively } 
\rho_2=(b,i_{\ell},d\mid {\zb}^-,\tau)\\
\zs_1=(c,i_{\ell},b\mid \tau,{\zb}) & (\textup{respectively } 
\zs_2=(d,i_{\ell},a\mid \tau^-,{\zb}^-),
\end{array}  \]
where $c$ is the starting point of $\tau $ and $d$ is its endpoint.

\begin{figure}[htp]
\resizebox{4.5in}{!}{\input{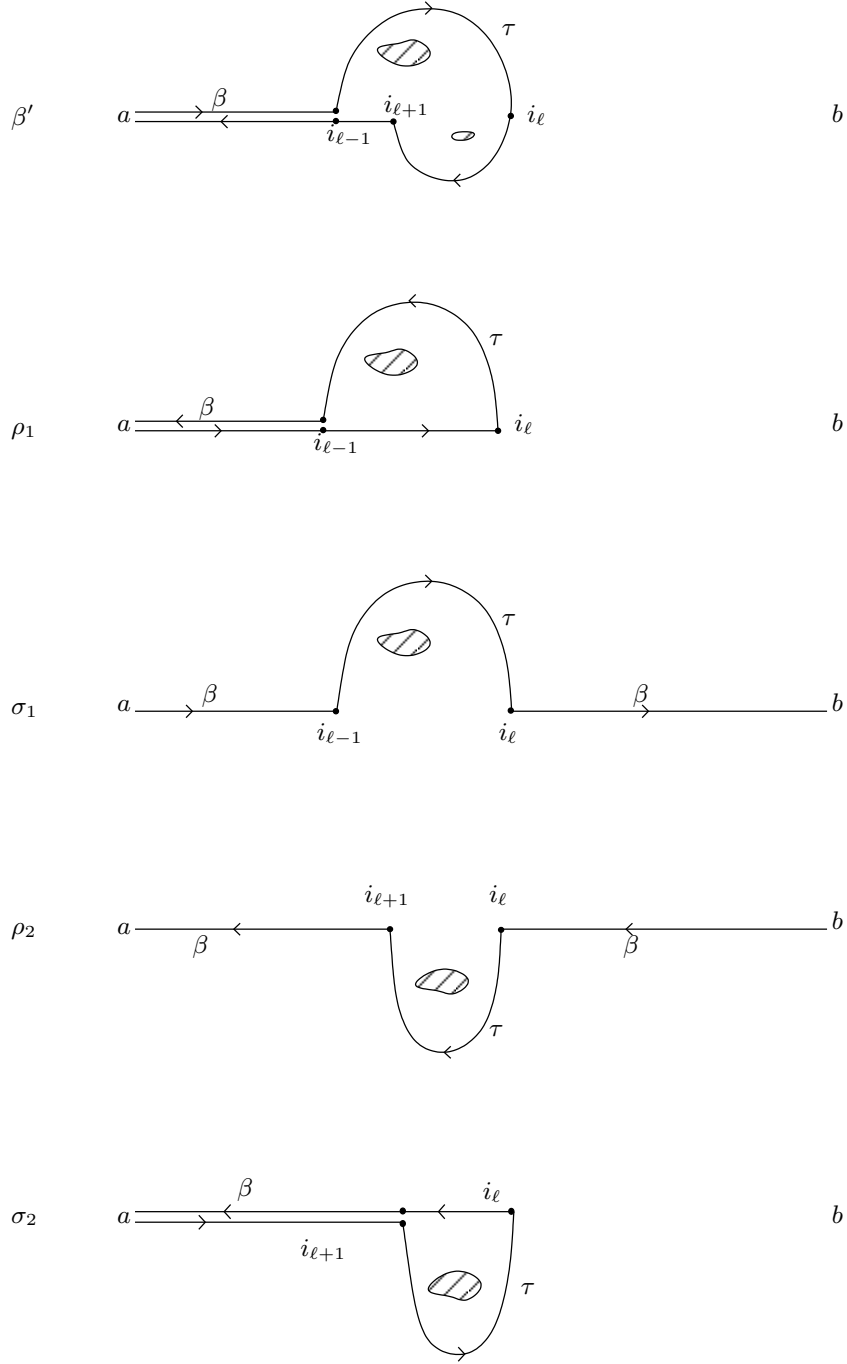}}
\caption{The arcs of Lemma \ref{lem exchange} in case where
  $i_{\ell-1}<i_\ell$ and $i_{\ell+1}<i_\ell$.}\label{fig arcs}
\end{figure}

Note that, since $i_{\ell-1},i_\ell$ and $i_{\ell+1}$ are distinct
crossing points of ${\zb} $ and $T$, the paths $(i_{\ell-1},i_\ell,
i_{\ell-1} \mid \tau,{\zb}^-)$ and  $(i_{\ell+1},i_\ell,
i_{\ell+1} \mid \tau^-,{\zb})$ are not homotopically trivial.

Then $\rho_1,\zs_1,\rho_2,\zs_2$ form a simply connected quadrilateral
such that 
$\rho_1$ and $\rho_2$ are opposite sides, $\zs_1$ and $\zs_2$ are
opposite sides, and ${\zb}$ and ${\zb'}$ are the diagonals. This shows (b).
Moreover, since ${\zb}$ and ${\zb'}$ cross exactly once, we have the
exchange relation in (c). It remains to show (a). Using the hypothesis 
$i_{\ell-1}<i_\ell$ and $i_{\ell+1}<i_\ell$, and the inequality
(\ref{ell}),  we get $i_{\ell-1}\le(k+1)/2, i_{\ell+1}\le(k+1)/2,
i_\ell+i_{\ell-1}\le k+1$ and $i_\ell+i_{\ell+1}\le k+1$. Thus

\[\begin{array}{rcl}
e({\zb'},T)&=&(i_{\ell-1}-1)+(i_{\ell+1}-1) < k \\
e(\rho_1,T)&=&(i_{\ell-1}-1)+(i_{\ell}-1) < k \\
e(\rho_2,T)&=&(i_{\ell+1}-1)+(k-i_{\ell}) < k+(i_{\ell+1}-i_\ell-1)<k \\
e(\zs_1,T)&=&(i_{\ell-1}-1)+(k-i_{\ell}) < k+(i_{\ell-1}-i_\ell-1)<k \\
e(\zs_2,T)&=&(i_{\ell+1}-1)+(i_{\ell}-1) < k. 
\end{array}  \]
 In the case where $\ell=1$, we have 
\[
\begin{array}{rcl}
e({\zb'},T)&=& i_{\ell+1}-1 < k\\
e(\rho_1,T)&=& i_{\ell}-1 < k\\
e(\zs_1,T)&=& k-i_{\ell} < k,
\end{array}  \]
and in the case where $\ell=r$, we have
\[
\begin{array}{rcl}
e({\zb'},T)&=& i_{\ell-1}-1 < k\\
e(\rho_2,T)&=&k-i_{\ell} < k\\
e(\zs_2,T)&=&  i_{\ell}-1 < k .
\end{array}  \]

This shows (a).\\


\item  \emph{$i_{\ell-1}<i_\ell$ and $i_{\ell+1}>i_\ell$.}  Consider the
  following   arcs, see Figure \ref{fig arcs2}: 

\[
\begin{array}{cc}
{\zb'}=(a,i_{\ell-1},i_{\ell+1},b\mid {\zb},\tau,{\zb})\\
 \rho_1=(a,i_{\ell-1},i_{\ell},a\mid {\zb},\tau,{\zb}^-) 
&\rho_2=(b,i_{\ell+1},i_{\ell},b\mid {\zb}^-,\tau^-,{\zb}) \\
 \zs_1=(a,i_{\ell},i_{\ell+1},b\mid {\zb},\tau,{\zb})
&\zs_2=(b,i_{\ell},i_{\ell-1},a\mid {\zb}^-,\tau^-,{\zb}^-).
\end{array}  
\]
In the special case where $\ell=1$, (respectively $l=r$), we define

\[
\begin{array}{cc}
{\zb'}=(c,i_{\ell+1},b\mid \tau,{\zb}) & (\textup{respectively } 
{\zb'}=(a,i_{\ell-1},d\mid {\zb},\tau)\\
\rho_1=(c,i_{\ell},a\mid \tau,{\zb}) & (\textup{respectively } 
\rho_2=(d,i_{\ell},b\mid \tau^-,{\zb})\\
\zs_2=(b,i_{\ell},c\mid {\zb},\tau) & (\textup{respectively } 
\zs_1=(a,i_{\ell},d\mid {\zb}^-,\tau^-),
\end{array}  \]
where $c$ is the starting point of $\tau $ and $d$ is its endpoint.

\begin{figure}[htp]
\resizebox{4.5in}{!}{\input{figarcs2.pstex_t}}
\caption{The arcs of Lemma \ref{lem exchange} in case where
  $i_{\ell-1}<i_\ell$   and
  $i_{\ell+1}>i_\ell$.}\label{fig arcs2}
\end{figure}

Again $\rho_1,\zs_1,\rho_2,\zs_2$ form a simply connected quadrilateral
such that 
$\rho_1$ and $\rho_2$ are opposite sides, $\zs_1$ and $\zs_2$ are
opposite sides, and ${\zb}$ and ${\zb'}$ are the diagonals. This shows (b).
Moreover, since ${\zb}$ and ${\zb'}$ cross exactly once, we have the
exchange relation in (c). It remains to show (a). Using  the hypothesis 
$i_{\ell-1}<i_\ell$ and $i_{\ell+1}>i_\ell$, and the  inequality
(\ref{ell}),
  we get $i_{\ell-1}\le(k+1)/2, k-i_{\ell+1}\le(k+1)/2,
i_\ell+i_{\ell-1}\le k+1$ and $i_\ell+i_{\ell+1}\ge k+1$. Thus

\[\begin{array}{rcl}
e({\zb'},T)&=&(i_{\ell-1}-1)+(k-i_{\ell+1}) < k \\
e(\rho_1,T)&=&(i_{\ell-1}-1)+(i_{\ell}-1) < k \\
e(\rho_2,T)&=&(k-i_{\ell+1})+(k-i_{\ell}) \le
2k-(i_\ell+i_{\ell+1})\le 2k-(k+1)<k\\
e(\zs_1,T)&=&(i_{\ell}-1)+(k-i_{\ell+1}) <k \\
e(\zs_2,T)&=&(k-i_{\ell})+(i_{\ell-1}-1) < k. 
\end{array}  \]
 In the case where $\ell=1$, we have 
\[
\begin{array}{rcl}
e({\zb'},T)&=& k-i_{\ell+1} < k\\
e(\rho_1,T)&=& i_{\ell}-1 < k\\
e(\zs_2,T)&=& k-i_{\ell} < k,
\end{array}  \]
and in the case where $\ell=r$, we have
\[
\begin{array}{rcl}
e({\zb'},T)&=& i_{\ell-1}-1 < k\\
e(\rho_2,T)&=& k-i_{\ell} < k\\
e(\zs_1,T)&=& i_{\ell}-1 < k.
\end{array}  \]

This shows (a).

\item \emph{$i_{\ell-1}>i_\ell$ and $i_{\ell+1}<i_\ell$.} This case
  follows from the case (2) by symmetry.

\item \emph{$i_{\ell-1}>i_\ell$ and $i_{\ell+1}>i_\ell$.} This case
  follows from the case (1) by symmetry.

\end{enumerate}   
\qed
\end{pf}

 Let $(S,M)$ be an unpunctured surface which is not simply
  connected. Let $T=\{\tau_1,\ldots,\tau_N\}$ be a triangulation of
  $S$, and let $\zg$ be an arc. Choose an orientation of $\zg$ and let
  $a$ be its starting point and $b$ its endpoint.
Denote by $\pi:\tilde S\to S$ a universal cover and define $\tilde
  T=\tau^{-1} (T)$. 

Fix a point $\tilde a \in \pi^{-1}(a)$ and let $\tilde \zg$ be the
unique lift of $\zg$ starting at $\tilde \za$. Although $\tilde T$ is
an infinite set, the set $ \mathcal{P}_{\tilde T}(\tilde \zg)$ is
finite, and we define the finite set
\[ \Tbar = \bigcup_{\tilde \za\in  \mathcal{P}_{\tilde T}(\tilde
  \zg) } 
\{\tilde\za_1,\tilde\za_2,\ldots,\tilde\za_{\ell(\tilde\za)}\}.
\]
Thus $\Tbar$ is the set of all arcs in $\tilde S$ that occur in some
path in $\mathcal{P}_{\tilde T}(\tilde \zg) $. 
Then $\Tbar$ is a triangulation of a simply connected surface
$(\Sbar,\Mbar)$: The boundary of $\Sbar$ is the union  of all arcs
that occur in some $\za \in \mathcal{P}_{\tilde T}(\tilde \zg) $ and
do not cross $\zg$, and the set of marked points
$\Mbar=\pi^{-1}(M)\cap \Sbar$. Consider the cluster algebra
$\mathcal{A}(\Sbar,\Mbar)$. By Theorem \ref{thm 1} for simply connected
surfaces, we have

\begin{equation}\label{eq pf}
x_{\tilde\zg}=\sum_{\overline\za\in \mathcal{P}_{\Tbar}(\tilde
  \zg)} x(\overline{\za}).
\end{equation}  
We want to apply $\pi$ to this equation in order to finish the
proof. Before we can do so, we need to study the effect of $\pi$ on
$(T,\zg)$-paths and on cluster variables.

\begin{lem}\label{lem pibar}
The covering map $\pi $ induces a bijection
\[\pibar:\mathcal{P}_{\Tbar}(\tilde \zg)\to \mathcal{P}_T(\zg)\]
 which sends a path
$\zabar=( \zabar_1,\zabar_2,\ldots,\zabar_{\ell(\zabar)})$ to the path
 $\pibar(\zabar)=
 ( \pi\circ\zabar_1,\pi\circ\zabar_2,\ldots,\pi\circ\zabar_{\ell(\zabar)}) $.
\end{lem}  
\begin{pf}
\emph{$\pibar $ is well defined}. Let $\zabar\in
\mathcal{P}_{\Tbar}(\tilde \zg) $. 
First note that the number of crossings of $\tilde\zg$ and $\Tbar$
equals the number of crossings of $\zg$ and $T$. Label the crossing
points in $\Sbar$ by $\overline{1},\overline{2},\ldots,\overline{k}$,
such that $\pi(\overline{i})=i$.

  $\pibar(\zabar)$ is a $T$-path. Indeed  $\pi\circ\zabar_i\in T$, since
$\zabar_i\in \Tbar$, and 
  $s(\pi\circ\zabar_i)=\pi(s(\zabar_i))=\pi(t(\zabar_{i-1}))=t(\pi\circ\zabar_{i-1})$,
  for $i=2,3,\ldots,\ell(\za)$. Moreover,
 $s(\pi\circ\zabar_1)=\pi(s(\zabar_1))=\pi(\tilde a)=a$, and 
$t(\pi\circ\zabar_{\ell(\zabar)})=\pi(t(\zabar_{\ell(\zabar)}))=\pi(\tilde
  b)=b$. 

$\pibar(\zabar)$ is reduced. Indeed,  suppose   $\tau=\pi\circ\zabar_i =
 (\pi\circ\zabar_{i-1})^-.$
Then $(\tau,\tau^-)=\pi\circ(\zabar_i,\zabar_{i-1})$ is a loop in $S$
 which is homotopically trivial. Therefore, its lift
 $(\zabar_i,\zabar_{i-1})$ is a  loop in $\Sbar$, whence
 $\zabar_i=(\zabar_{i-1})^-$, a contradiction.

It remains to  to check the axioms
(T1)--(T5) for $\pibar(\zabar)$. 
\begin{itemize}
\item[(T1)] $\ell(\pibar(\zabar))=\ell(\zabar)$ is odd.
\item[(T2)] If $i$ is even, $\zabar_i$ crosses $\tilde \zg$, hence
  $\pi\circ\zabar_i$ crosses $\zg$.
\item[(T3)] 
The number $e(\zg,\tau_i)$ of crossings between $\zg $ and $\tau_i$ in
$S$ is equal to the number of crossings
between $\tilde \zg$ and $\pi^{-1}(\tau_i)$ in $\Sbar$.

Now,  $\pi\circ\zabar_t=\tau_i$  if and only if
$\zabar_t\in\pi^{-1}(\tau_i)\cap \Sbar$. 
Therefore, since $\zabar $ satisfies
  condition (T3), the number of even  integers $t$ such that
  $\pi\circ\zabar_t=\tau_i$  is
 at most the number of crossings between $\tilde\zg$ and
 $\pi^{-1}(\tau_i)$ in $\Sbar$, hence  the number of even  integers
 $t$ such that   $\pi\circ\zabar_t=\tau_i$  is  at most $e(\zg,\tau_i)$.
\item[(T4)] Since $\zabar$ satisfies condition (T4), there exists a subsequence
  $(\overline{i_2},\overline{i}_4,\ldots,\overline{i}_{\ell(\zabar)-1})$
  of $(\overline{1},\ldots,\overline{k})$ of crossing points in
  $\Sbar$ such that the crossing point $\overline{i}_j$ lies on
  $\zabar_{j}$ and $\overline{i}_j\ne \overline{i}_\ell$ if $j\ne \ell$.
Let $i_j$ be the the image of the crossing point $\overline{i}_j$
  under $\pi$.
 Then
  $(i_2,i_4,\ldots, i_{\ell(\zabar)-1})$ is a subsequence of
  $(1,2,\ldots,k)$ in $S$ such that the crossing point $i_j$ lies on
  $\pi\circ\zabar_{i_j}$ and  $i_j\ne i_\ell$ if $j\ne \ell$.
\item[(T5)] Two paths
  $\za$ and $\zb$ in $S$ are homotopic if and only if their lifts
  $\tilde\za $ and $ \tilde{ \zb}$, which start at same point in
  $\Sbar$,   are homotopic.
This implies (T5).
\end{itemize}   
\emph{$\pibar $ is injective.} Suppose $\pibar(\zabar)=\pibar(\zbbar)$,
that is, 
$\pi\circ\zabar_i=\pi\circ\zbbar_i$, for all $i$. In particular,
$\ell(\zabar)=\ell(\zbbar)$. Now, $\zabar_i$ is the unique lift of
$\pi\circ\zabar_i$ that starts at $s(\zabar_i)$, and
 $\zbbar_i$ is the unique lift of
$\pi\circ\zbbar_i=\pi\circ\zabar$ that starts at
$s(\zbbar_i)$, for $i=1,\ldots,\ell(\zabar)$. 
Moreover, $s(\zabar_1)=\tilde a=s(\zbbar_1)$. Consequently,
$\zabar=\zbbar$ and $\pibar$ is injective.

\emph{$\pibar $ is surjective.} 
For every $\za\in \mathcal{P}_T(\zg) $ there is a  lift $\tilde
\za$ that starts at $\tilde a$. We have to show that $\tilde \za \in
\mathcal{P}_{\Tbar}(\tilde \zg) $. (T1) and (T2) follow directly from the
construction, and (T5) holds since $\Sbar $ is simply connected. 

Let 
$\za=(\za_1,\ldots,\za_{\ell(\za)})$ and
$\tilde\za=(\tilde \za_1,\ldots,\tilde \za_{\ell(\za)})$. 
Suppose  there exist even integers $s $ and $t$  such that
$\tilde\za_s=\tilde \za_t$. Then the arc $\za_s=\za_t$ crosses $\zg$
in  the two points $i_s$ and $i_t$.
Suppose without loss of generality that $i_t<i_s$. Then the loop 
$\za^\circ=(i_s,i_t,i_s\mid \za_s,\zg)$ is the image of a loop
$\tilde{\za}^\circ=(\overline{i}_s,\overline{i}_t,\overline{i}_s\mid
\zabar_s,\zgbar)$ in $\Sbar$ under $\pibar$.
Since $\pi:\tilde S \to S$ is a universal cover, $\za^\circ$ is
homotopically trivial, and therefore there is an arc in the isotopy
class of $\za_s$ which crosses $\zg$ two times fewer than $\za_s$, a
contradiction. 

This shows that the even arcs of $\tilde \za $ are pairwise
disjoint. Therefore, $\tilde\za$ satisfies (T3) and (T4).
\qed
\end{pf}

By the Laurent phenomenon \cite{FZ1}, every element of the cluster
algebra $\mathcal{A}(\Sbar,\Mbar)$ is a 
Laurent polynomial in the cluster variables  $\{x_{\taubar} \mid
\taubar\in\Tbar\} $.
Therefore,
$\pi $ induces a homomorphism of algebras
$\pi_\mathcal{A}:\mathcal{A}(\Sbar,\Mbar)\to \mathcal{A}(S,M)$ which
is given on the generators $x_{\taubar}$, $\taubar\in \Tbar$ by 
\[\pi_\mathcal{A}(x_{\taubar})=x_{\pi\circ\taubar}.\]

\begin{lem}\label{lem pi} 
Let $\tilde\zb$ be an  arc in $(\Sbar,\Mbar)$ which is a lift of an
arc $\zb$ in $(S,M)$. Then 
\[\pi_\mathcal{A}(x_{\tilde\zb})=x_\zb \]
\end{lem}  

\begin{pf}
We prove the Lemma by induction on $k=e(\zb,T)$, the minimal number  of
crossing points between $\zb $ and $T$.  
If $k=0$, then $\zb\in T$, and $\pi_\mathcal{A}(x_{\tilde\zb}) =x_\zb$, by
definition. 
Suppose that $k\ge 1$. Let $\rho_1,\ \rho_2,\ \zs_1,\ \zs_2,$
and $\zb'$ be as in Lemma \ref{lem exchange}. Then, in
$\mathcal{A}(S,M)$, we have the
exchange relation 
\begin{equation}\label{exchange}
 x_{\zb} =(  x_{\rho_1} x_{\rho_2}+ x_{\zs_1}x_{ \zs_2})/ x_{\zb'}.
\end{equation}

Let $\tilde \rho_1$
and $\tilde \zs_2$ be the unique lifts  of $\rho_1$ and $\zs_2$,
respectively,  that start at the same point as $\tilde \zb$. 
Let $\tilde \zs_1$ and
$\tilde {\zb'}$ be the unique lifts of $\zs_1$ and $\zb'$, respectively, that start at
the endpoint of $\tilde \rho_1$, and let $\tilde \rho_2$ be the unique
lift of $\rho_2$ that starts at the endpoint of $\tilde\zs_2$.
Then $\tilde \rho_1,\,\tilde \rho_2, \,\tilde \zs_1, \,\tilde \zs_2$
form a quadrilateral in $\Sbar$ in which $\tilde \zb$ and $\tilde {\zb'}$
are the diagonals. Consequently, in the cluster algebra
$\mathcal{A}(\Sbar,\Mbar)$, we have the exchange relation
\begin{equation}\nonumber
 x_{\tilde \zb} =(  x_{\tilde \rho_1} x_{\tilde \rho_2}+ x_{\tilde
  \zs_1}x_{\tilde \zs_2})/ x_{\tilde {\zb'}}.
\end{equation}   

Therefore 
\begin{equation}\nonumber
 \pi_\mathcal{A} (x_{\tilde \zb} )=
( \pi_\mathcal{A}  (x_{\tilde \rho_1}) \pi_\mathcal{A}( x_{\tilde
   \rho_2})
+  \pi_\mathcal{A}(x_{\tilde  \zs_1})  \pi_\mathcal{A}(x_{\tilde
  \zs_2}))/  \pi_\mathcal{A}(x_{\tilde {\zb'}}),
\end{equation}   
 and by induction,
\begin{equation}\nonumber
 \pi_\mathcal{A} (x_{\tilde \zb} )=
( x_{\rho_1} x_{\rho_2}
+x_{\zs_1}  x_{  \zs_2})/ x_{{\zb'}},
\end{equation}   
which is equal to $x_{\zb}$ by equation (\ref{exchange}).
\qed
\end{pf}

\emph{Proof of Theorem \ref{thm 1}:} 
By equation (\ref{eq pf}), we have 
\[x_{\tilde\zg}=\sum_{\overline\za\in \mathcal{P}_{\Tbar}(\tilde
  \zg)} x(\overline{\za}).
\]
We apply the homomorphism $\pi_{\mathcal{A}}$ to this equation,
using Lemma \ref{lem pi}, and we get

\[x_{\zg}
=\sum_{\overline\za\in \mathcal{P}_{\Tbar}(\tilde
  \zg)} x(\pibar(\overline{\za})).
\]
And by Lemma \ref{lem pibar}, this implies 
\[x_{\zg}
=\sum_{\za\in \mathcal{P}_{T}(\zg)} x(\za).
\]
\qed
\end{subsection}

\end{section}


\begin{section}{Expansion formula using mutant arcs}\label{sect 3}

For each non-boundary arc $\alpha$ of $T$, define $\alpha'$ to be the 
arc obtained by mutating $T$ at $\alpha$.  
We will call these ``mutant arcs''.  Let $M$ be the set of these
arcs.  

Fix an arc $\gamma$ which is not in $T$.  
In this section, we will give a new expression for
$x_\gamma$.  

Let $X_\gamma$ be the set of arcs which 
$\gamma$ crosses, or which 
bound a triangle which $\gamma$ crosses.  
Let 
$M_\gamma$ be the mutant arcs corresponding to arcs which $\gamma$ crosses.  

For every triangle through which $\gamma$ passes, 
other than the first and last
triangles, $\gamma$ passes through two of the three arcs of the triangle. 
Let $E_{\gamma}$ 
be the set of the third arcs of these triangles, that is to say,
for each of these triangles, the arc through which $\gamma$ does not pass.

We have already shown that $x_\gamma$ can be written as a Laurent
polynomial in $x_\sigma$ for $\sigma\in X_\gamma$.  

In this section, we will give an explicit expression for 
$x_\gamma$
as a Laurent polynomial in $x_\sigma$ for $\sigma \in X_\gamma\cup M_\gamma$, 
where only variables $x_\sigma$ with $\sigma\in E_\gamma$ occur in 
the denominator.  

\begin{subsection}{Connection to Cluster Algebras} 
The motivation for this result is that it is an explicit 
version of a result from \cite{CA3}.   

The matrix $B$ arising from a cluster is called {\it acyclic} if it
satisfies a certain  
condition, which, in our setting, is equivalent to the condition 
that every
triangle of $T$ has at least one arc lying along the boundary.  
  
Define $x'_i$ to be the 
cluster variable obtained by mutating $x_i$.  
Theorem 1.20 of  
\cite{CA3} says that, if $B$ is acyclic, any cluster variable
in the cluster algebra has an expression as a polynomial 
$f\in \mathbb{ZP}[x_1,\ldots,x_n,x'_1,\ldots,x_n']$. 
%
%

Let us interpret this statement in our setting.  
Let $E$ be the set of boundary arcs.
Theorem 1.20 of  
\cite{CA3} tells us precisely that if 
every triangle of $T$ has an arc on the boundary, then 
there is
 an expression for an arbitrary $x_\gamma$ as a Laurent polynomial
in the variables $x_\sigma$ with $\sigma\in T\cup M$, with
only variables $x_\sigma$ with $\sigma\in E$ appearing in the 
denominator.  Since $E\supset E_\gamma$, $T\supset X_\gamma$, and $M\supset
M_\gamma$, our result is an
explicit realization of the kind of representation guaranteed by
Theorem 1.20.    
\end{subsection}

\begin{subsection}{Recurrence}
Let $A,B$ be marked points, and let $\gamma$ be an arc from $A$ to $B$.
As established in the previous section, if we want to calculate
$x_\gamma$, we can ``unwind'' $\gamma$ and reduce ourselves to the 
simply connected case, where $A$ and $B$ are marked points on the boundary
of a polygon, $P_{AB}$, through all of whose triangles the line segment
$AB$ passes.  We will now write $E_{AB}$ instead of $E_\gamma$.  

For any two marked points $C$ and $D$ on the boundary of 
$P_{AB}$, we write $P_{CD}$ for the polygon
consisting of the union of the triangles through which the line from
$C$ to $D$ passes. 

We will let $I_{AB}$ be the set of interior arcs of $P_{AB}$.  We will speak
of the $I_{AB}$-degree of a vertex, meaning the number of $I_{AB}$-arcs 
incident
with it.  

Suppose that $P_{AB}$ has at least five vertices.  
Let
$V_0$ be the vertex adjacent to $A$ whose $T$-degree is greater than one.
Let $V_1$ be the other vertex adjacent to $A$, and let $V_2$ be the next
vertex along the boundary from it, see Figure \ref{fig mut1} or
\ref{fig mut2} for the two possible configurations.   

\begin{lem}\label{lem:rec} We can calculate the cluster variable $x_{AB}$ by
means of the following recurrence:

\begin{equation}\label{rec}
 x_{AB}=\frac{x_{AV_2}x_{V_1B}-x_{AV_1}x_{V_2B}}{x_{V_1V_2}}
\end{equation}
\end{lem}

\begin{proof} The result follows immediately from the exchange relation:

$$x_{AV_2}x_{V_1B}=x_{AV_1}x_{V_2B}+x_{AB}x_{V_1V_2}.$$

\end{proof}

The above recursion is, in fact, all we need to recover Theorem 1.20 of
\cite{CA3} in our setting, as the following corollary shows.  

\begin{cor} There is an expression for $x_{AB}$ as a Laurent polynomial
in the variables $x_\sigma$ for $\sigma \in X_{AB}\cup M_{AB}$, such that the only
variables appearing in the denominator are $x_\sigma$ with $\sigma\in E_{AB}$.
\end{cor}

\begin{proof} The proof is by induction on the number of arcs of $T$
crossed by $AB$.  If $AB$ crosses no arcs of $T$, it is in $T$, so 
$x_{AB}$ is the desired expression.  If $AB$ crosses exactly one arc of $T$, then
it is a mutant arc, so $x_{AB}$ is, again, an expression of the desired form.

So assume that $AB$ crosses at least two arcs of $T$, and that the
claim is proved for any $x_{CD}$ such that $CD$ crosses fewer 
arcs of $T$ than $AB$ does.  Since $AB$ crosses at least two arcs of $T$,
we know $P_{AB}$ has at least five vertices, and we can apply 
Lemma~\ref{lem:rec}.
Let $V_0, V_1, V_2$ be as in the preamble to Lemma~\ref{lem:rec}.  So (\ref{rec}) holds.
Note that the only variable in the denominator of the right hand
side of (\ref{rec}) is 
$x_{V_1V_2}$, and $V_1V_2$ is in $E_{AB}$.  
Of the cluster variables which appear in
the numerator, $x_{AV_1}$ is from $T$, and 
$x_{AV_2}$ is from $M$.  This leaves $x_{V_1B}$ and $x_{V_2B}$.  The
corresponding arcs are typically from neither $X$ nor $M$, but they cross
fewer arcs of $T$ than does $AB$, so the induction hypothesis applies,
and we know that $x_{V_iB}$ can be written as a Laurent polynomial in 
$x_\sigma$ with $\sigma \in T\cup M$, where the variables appearing in 
the denominator are from $E_{V_iB}$.  Since $E_{V_iB}\subset E_{AB}$,
the expression guaranteed by the induction hypothesis for $x_{V_iB}$ is of
the desired form, and the corollary is proved.
\end{proof}
\end{subsection}

\begin{subsection}{Explicit formula}
We will now proceed to give an explicit formula of the form guaranteed
by Theorem 1.20 of \cite{CA3} or by the previous corollary.  
In order to do this, we need to introduce some further notation.  
Number the arcs of $I_{AB}$ as $I_{AB}^0$, $I_{AB}^1,$ etc., 
in the order in which they cross
$AB$, and then likewise number the corresponding mutant arcs so that 
the mutation of the arc of $I_{AB}^i$ is $M_{AB}^i$.
The arc $M_{AB}^i$ is considered to be oriented so that
it crosses $I_{AB}^i$ from the side on which $A$ lies towards
the side on which $B$ lies.  The boundary arcs of $P_{AB}$ are oriented
to point from $A$ to $B$.  

Let the arcs $\overline I_{AB}$ be the arcs of $I_{AB}$ which connect 
vertices whose $I_{AB}$-degree is at least two.  

We now define a set of paths from $A$ to $B$, which we denote 
$\mathcal G_{AB}$.  These are the paths satisfying the following properties:
\begin{enumerate}
\item[G1$_{AB}$] The length of the path is odd.
\item[G2$_{AB}$] The arcs appearing are from the boundary of 
$P_{AB}$, $\overline I_{AB}$, and
$M_{AB}$.
\item[G3$_{AB}$] The arcs appearing in even position are all from
$E_{AB}$.
\item[G4$_{AB}$] The arcs of $\overline I_{AB}\cup M_{AB}$ 
are used in numerical order
(so in particular, at most one of $I_{AB}^i$ and $M_{AB}^i$ is used).
\item[G5$_{AB}$] The arcs of $M_{AB}$ are used only in the forward
direction.  
\item[G6$_{AB}$] The path
 should not touch $A$ other than at the beginning, and not touch $B$
other than at the end.
\end{enumerate}

For $\gamma$ an odd-length path between two vertices $C$ and $D$,  
define $k(\gamma)$ 
to be the number of boundary 
arcs of $\gamma$ which are used contrary to their orientation, plus 
$(\ell(\gamma)-1)/2$, plus the number of arcs of $I_{CD}$ used.  

For $\mathcal A$ a set of paths, write $\int \mathcal A$ for 
$\sum_{\gamma\in\mathcal A}(-1)^{k(\gamma)}x(\gamma)$.  
Now we have the following theorem:

\begin{thm}\label{thm 2} We have the following expression for the cluster variable
$x_{AB}$:
\begin{equation}\label{eq:main} x_{AB}=\int \mathcal G_{AB} \end{equation}
\end{thm}

\begin{proof} The proof is by induction.  The theorem is clearly true if 
$AB$ is an arc of $T$, or $AB$ is a mutant arc.  So assume that $AB$
crosses at least two arcs of $T$, and that the theorem holds for any
segment $CD$ which crosses fewer arcs of $T$ than $AB$ does.  

The proof involves showing that the formula (\ref{eq:main}) satisfies
the recurrence in Lemma~\ref{lem:rec}.  
By the induction hypothesis, 

$$x_{V_iB}=\int \mathcal G_{V_iB}.$$

By Lemma~\ref{lem:rec}, we therefore know:

$$x_{AB}=\frac{x_{AV_2}}{x_{V_2V_1}}\int\mathcal G_{V_1B} - \frac{x_{AV_1}}
{x_{V_1V_2}}\int \mathcal G_{V_2B}.$$

We will now establish that the right hand side of the above equation equals
$\int \mathcal G_{AB}$.  
We split into two cases, depending on whether the $I_{AB}$-degree of $V_0$ is 2 
or more.  

Suppose first that it is 2.  So our situation is as in Figure \ref{fig
  mut1}.
We have labelled the next vertex after $V_0$ as $V_3$. 
\begin{figure}
\begin{center}
\input{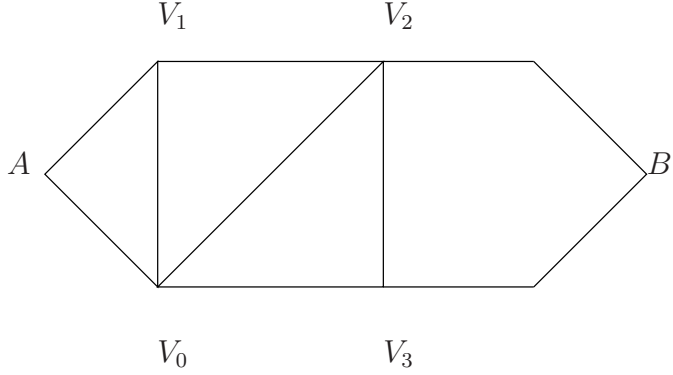}
\caption{First case,  $I_{AB}$-degree of $V_0$ is 2}\label{fig mut1}
\end{center}
\end{figure}

For $i=1,2,3$, write $\mathcal F_i$ for the paths from
$V_i$ to $B$ but 
satisfying the conditions G1$_{AB}$--G6$_{AB}$.   
(In particular, no $\mathcal F_i$ uses 
the arc $V_0V_1$, since it is an internal arc
not in $\overline I_{AB}$.)    
So $\mathcal G_{V_1B}$ includes
paths in addition to those in $\mathcal F_1$, because $\mathcal G_{V_1B}$
can include paths using the arc $V_0V_1$.

Write $V_1V_0\mathcal F_3$ for paths that run from $V_1$
to $V_0$ to $V_3$, and thence follow a path in 
$\mathcal F_3$, and similarly for other (even-length) 
sequences of vertices followed by a set of paths.  Then

$$\mathcal G_{V_1B}= \mathcal F_1 \amalg V_1V_0\mathcal F_3.$$  

Now, 

$$\frac {x_{AV_1}}{x_{V_1V_2}}\int\mathcal G_{V_2B}= -\int AV_1\mathcal G_{V_2B}$$
where the negative sign appears because the lengths of the paths being
summed over have each decreased by two.  And similarly,

\begin{equation}\label{eqn1}
\frac{x_{AV_2}}{x_{V_2V_1}}\int\mathcal G_{V_1B}=
\int AV_2\mathcal F_1 - \int  AV_2V_1V_0\mathcal F_3
\end{equation}
where the first term is positive because the change in length is cancelled
out by adding a backwards boundary arc, and the 
negative sign in the second term appears because, in addition to the two
effects already mentioned, we have an arc $V_1V_0$ which was formerly a
boundary arc, but is now in $I_{AB}$, so contributes a factor of -1.  

We must simplify further, however, because the paths in the second term are
not in $\mathcal G_{AB}$, as they use both $V_0V_1$ and its mutant arc
$AV_2$, so this violates G4$_{AB}$.  
Thus, we must take advantage of the exchange
relation which tells us that

$$x_{AV_2}x_{V_1V_0}=x_{V_1V_2}x_{AV_0}+x_{AV_1}x_{V_0V_2}.$$

Thus, 

$$(\ref{eqn1})=\int AV_2\mathcal F_1 +\int AV_1V_2V_0\mathcal F_3
+\int AV_0\mathcal F_3$$

Since
$$\mathcal G_{AB}=AV_2\mathcal F_1 \amalg AV_1V_2V_0\mathcal F_3
\amalg AV_0\mathcal F_3 \amalg AV_1\mathcal G_{V_2B},$$
it follows that $x_{AB}=\int \mathcal G_{AB}$, as desired.

We now consider the second case, where the situation is shown in
Figure \ref{fig mut2}.

\begin{figure}
\begin{center}
\input{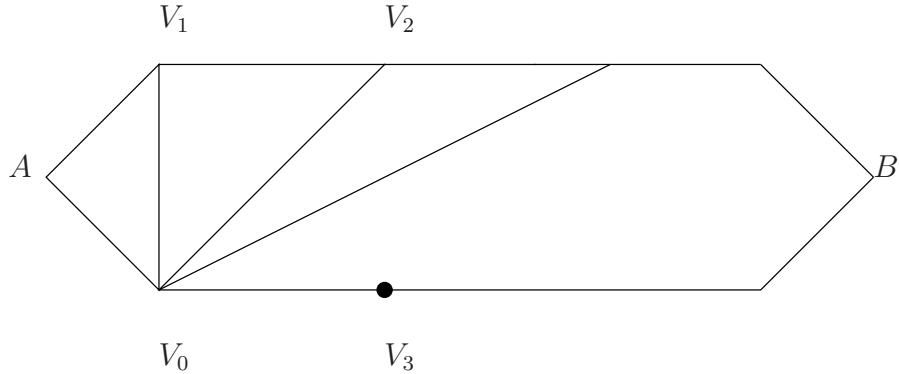}
\caption{Second case,  $I_{AB}$-degree of $V_0$ is greater than 2}\label{fig mut2}
\end{center}
\end{figure}

 Note that the next vertex after $V_0$ away from $A$ has been labelled
$V_3$.  In this case, the analysis is similar to the previous case.

$$\mathcal G_{V_1B}=\mathcal F_1 \amalg V_1V_0\mathcal F_3$$
$$\mathcal G_{V_2B}=\mathcal F_2\amalg V_2V_0\mathcal F_3$$
$$\frac{x_{AV_1}}{x_{V_1V_2}}\int \mathcal G_{V_2B}= 
-\int AV_1\mathcal F_2 + \int AV_1V_2V_0\mathcal F_3$$
$$\frac{x_{AV_2}}{x_{V_2V_1}}\int \mathcal G_{V_1B} = \int AV_2\mathcal F_1
-\int AV_2V_1V_0\mathcal F_3 = \int AV_2\mathcal F_1 + 
\int AV_1V_2V_0\mathcal F_3 + \int AV_0\mathcal F_3$$

Since $\mathcal G_{AB}= AV_2\mathcal F_1 \amalg AV_0\mathcal F_3 \amalg
AV_1\mathcal F_2$, it follows that 
$x_{AB}=\int \mathcal G_{AB}$, as desired.
\end{proof}
\end{subsection}
\end{section}

{} 

\vspace{1cm}

\noindent 
Ralf Schiffler\\
Department of  Mathematics and    Statistics\\
University of   Massachusetts at Amherst\\
Amherst, MA 01003--9305, USA\\ 
E-mail address: {\tt schiffler@math.umass.edu}\\

\noindent  Hugh Thomas\\ 
Department of Mathematics and
Statistics\\University of New Brunswick\\
Fredericton, New Brunswick, E3B 5A3
CANADA\\ 
E-mail address: {\tt hugh@math.unb.ca}

\end{document}